\documentclass[11pt, a4paper]{amsart}														
\usepackage{amssymb,amsfonts,amsmath,amsthm,amsopn,amstext,amscd,mathtools}
\usepackage[colorlinks=true, citecolor=magenta, linkcolor=blue, urlcolor=blue1, filecolor=cyan]{hyperref}
\usepackage[margin=1in]{geometry}
\usepackage{xy}
\xyoption{all}

\usepackage[usenames,dvipsnames]{xcolor}
\definecolor{blue}{RGB}{51,103, 255}

\usepackage{tabularx}
\usepackage{latexsym}

\usepackage[only, llbracket,rrbracket]{stmaryrd}
\usepackage{todonotes}

\usepackage[shortlabels]{enumitem} 
\usepackage{bbm}
\usepackage{pdfpages}
\usepackage{titletoc}
\usepackage{booktabs}
\usepackage{mathtools}
\usepackage{mathrsfs}
\usepackage{thmtools}
\usepackage{hhline}
\usepackage{caption}
\usepackage{multirow, url}
\usepackage{textcomp}
\DeclareMathAlphabet{\cmcal}{OMS}{cmsy}{m}{n}

\usepackage{arydshln}
\usepackage{libertine}
\usepackage[utf8]{inputenc}
\usepackage[OT2,T1]{fontenc}

\DeclareSymbolFont{euletters}{U}{eur}{m}{n}
\DeclareSymbolFont{eufrakletters}{U}{euf}{m}{n}
\DeclareFontFamily{U}{wncy}{}
    \DeclareFontShape{U}{wncy}{m}{n}{<->wncyr10}{}
    \DeclareSymbolFont{mcy}{U}{wncy}{m}{n}
    \DeclareMathSymbol{\Sha}{\mathord}{mcy}{"58}

\newtheorem{thm}{Theorem}[section]
\newtheorem{cor}[thm]{Corollary}
\newtheorem{lem}[thm]{Lemma}
\newtheorem{prop}[thm]{Proposition}

\newtheorem{conj}[thm]{Conjecture}

\theoremstyle{defn}
\newtheorem{defn}[thm]{Definition}

\theoremstyle{remark}
\newtheorem{rem}[thm]{\bf{Remark}}
\newtheorem{notation}[thm]{\bf{Notation}}

\numberwithin{equation}{section} \numberwithin{table}{subsection}
	
\newtheorem*{theorem*}{\bf{Theorem}}
\newtheorem*{claim*}{\bf{Claim}}
\newtheorem*{remark*}{\bf{Remark}}
\newtheorem*{remarks*}{\bf{Remarks}}
\newtheorem*{example*}{\bf{Example}}
\newtheorem*{examples*}{\bf{Examples}}

\newcommand{\C}{{\mathbf{C}}}

\newcommand{\Q}{{\mathbf{Q}}}

\newcommand{\Z}{{\mathbf{Z}}}

\newcommand{\scC}{{\mathscr{C}}}

\newcommand{\scS}{{\mathscr{S}}}

\newcommand{\scU}{{\mathscr{U}}}

\newcommand{\tn}{\textnormal}

\def\a{\alpha}

\def\d{\delta}
\def\e{\epsilon}
\def\ve{\varepsilon}
\def\g{\gamma}

\def\s{\sigma}
\def\z{\zeta}
\def\p{\varphi}

\def\G{\Gamma}

\def\<{\left\langle}
\def\>{\right\rangle}
\def\gcd{\textnormal{gcd}}

\def\min{\textnormal{min}}

\newcommand{\zmod}[1]{{\Z/{#1}\Z}}

\newcommand{\arinj}{\ar@{^(->}}
\newcommand{\arsurj}{\ar@{->>}}
\newcommand{\arsub}{\ar@{}[r]|-*[@]{\subset}}
\newcommand{\arsup}{\ar@{}[r]|-*[@]{\supset}}
\newcommand{\arcap}{\ar@{}[d]|-*[@]{\subset}}
\newcommand{\arcup}{\ar@{}[u]|-*[@]{\subset}}
\newcommand{\arin}{\ar@{}[u]|-*[@]{\in}}

\renewcommand{\pmod}[1]{{\,(\textnormal{mod}\hspace{1mm} {#1})}}
\renewcommand{\mod}[1]{{\,\textnormal{mod}\hspace{1mm} {#1}}}

\newcommand{\Gal}{{\textnormal{Gal}}}

\newcommand{\SL}{{\textnormal{SL}}}

\newcommand{\sHom}{{\mathscr{H}\kern-.5pt om}}
\newcommand{\sExt}{{\mathscr{E}\kern-.5pt xt}}

\newcommand{\tor}{{\textnormal{tors}}}

\newcommand{\lcm}{{\textnormal{lcm}}}

\renewcommand{~}{\hspace*{0.5mm}}

\newcommand{\ov}{\overline}

\newcommand{\ms}{\medskip}
\newcommand{\sm}{\smallsetminus}

\mathchardef\hyp="2D

\newcommand{\gauss}[1]{\left\lfloor #1 \right\rfloor}

\newcommand{\qa}{{\quad \text{and} \quad}}
\newcommand{\qqa}{{~ \text{ and } ~}}

\newcommand{\qqo}{{~ \text{ or } ~}}

\newcommand{\mat}[4]{
 \left(  \begin{smallmatrix} #1 & #2 \\ #3 & #4 \end{smallmatrix} \right)}

\newcommand{\vect}[2]{
 \left(  \begin{smallmatrix} #1 \\ #2 \end{smallmatrix} \right)}

\newcommand{\leg}[2]{\genfrac(){}{}{#1}{#2}}

\newcommand{\hide}[1]{}
\newcommand{\dd}{~|~}

\begin{document}

\title{The rational cuspidal subgroup of $J_0(p^2M)$ with $M$ squarefree}

\author{Jia-Wei Guo}
\address{Jia-Wei Guo, Department of Mathematics, National Taiwan University, Taipei 10617, Taiwan}
\email{\textnormal{jiaweiguo312@gmail.com}}

\author{Yifan Yang}
\address{Yifan Yang, Department of Mathematics, National Taiwan University and National Center for Theoretical Science, Taipei 10617, Taiwan}
\email{\textnormal{yangyifan@ntu.edu.tw}}

\author{Hwajong Yoo}
\address{Hwajong Yoo, College of Liberal Studies and Research Institute of Mathematics, Seoul National University, Seoul 08826, South Korea}
\email{\textnormal{hwajong@snu.ac.kr}}                                                                  

\author{Myungjun Yu}
\address{Myungjun Yu, Department of Mathematics, Yonsei University, Seoul 03722, South Korea}
\email{\textnormal{mjyu@yonsei.ac.kr}}      
                                                           
\maketitle
\begin{abstract}
For a positive integer $N$, let $X_0(N)$ be the modular curve over $\Q$ and $J_0(N)$ its Jacobian variety. We prove that the rational cuspidal subgroup of $J_0(N)$ is equal to the rational cuspidal divisor class group of $X_0(N)$ when $N=p^2M$ for any prime $p$ and any squarefree integer $M$. To achieve this we show that all modular units on $X_0(N)$ can be written as products of certain functions $F_{m, h}$, which are constructed from generalized Dedekind eta functions. Also, we determine the necessary and sufficient conditions for such products to be modular units on $X_0(N)$ under a mild assumption. 
\end{abstract}
\setcounter{tocdepth}{1}
\tableofcontents


\section{Introduction}
For a positive integer $N$, let $X_0(N)$ be the canonical model over $\Q$ of the modular curve associated to the congruence subgroup 
\begin{equation*}
\G_0(N)=\left\{ \mat a b c d  \in \SL(2, \Z) : c \equiv 0 \pmod N\right\},
\end{equation*}
and let $J_0(N)$ be its Jacobian variety. We are interested in the rational torsion subgroup of $J_0(N)$, denoted by $J_0(N)(\Q)_\tor$. When $N$ is a prime, Mazur proved the following \cite[Th. 1]{M77}, which was known as Ogg's conjecture \cite[Conj. 2]{Og75}:
\begin{thm}[Ogg's conjecture]
The rational torsion subgroup $J_0(N)(\Q)_\tor$ is a cyclic group of order $n$, generated by the linear equivalence class of the difference of the two cusps $0$ and $\infty$, where $n$ is the numerator of $\frac{N-1}{12}$.
\end{thm}

In order to generalize this theorem, we first introduce two subgroups $\scC(N)$ and $\scC_N(\Q)$ of $J_0(N)(\Q)_\tor$. Recall that a divisor on $X_0(N)$ is called \textit{cuspidal} if its support lies on the cusps of $X_0(N)$. Also, a cuspidal divisor $D$ is called \textit{rational} if $\s(D)=D$ for all $\s \in \Gal(\ov{\Q}/\Q)$. 
Let $\scC(N)$ be the \textit{rational cuspidal divisor class group} of $X_0(N)$, defined as a subgroup of $J_0(N)(\Q)$ generated by the linear equivalence classes of rational cuspidal divisors of degree $0$ on $X_0(N)$. Also, let $\scC_N$ be the \textit{cuspidal subgroup} of $J_0(N)$, defined as a subgroup of $J_0(N)(\ov{\Q})$ generated by the linear equivalence classes of cuspidal divisors of degree $0$ on $X_0(N)$. Furthermore, let 
\begin{equation*}
\scC_N(\Q):=\scC_N \cap J_0(N)(\Q)
\end{equation*}
be the \textit{rational cuspidal subgroup} of $J_0(N)$. Note that we have
\begin{equation*}
\scC(N) \subseteq \scC_N(\Q) \subseteq J_0(N)(\Q)_\tor,
\end{equation*}
where the first (resp. second) inclusion follows by definition (resp. the theorem of Manin and Drinfeld \cite{Ma72, Dr73}). A natural generalization of the above theorem is the following.

\begin{conj}[Generalized Ogg's conjecture]\label{conjecture: GOC}
For any positive integer $N$, we have
\begin{equation*}
J_0(N)(\Q)_\tor = \scC_N(\Q).
\end{equation*}
\end{conj}

We also expect the following, which is \cite[Conj. 1.3]{Yoo1}.
\begin{conj}\label{conjecture}
For any positive integer $N$, we have
\begin{equation*}
\scC_N(\Q)=\scC(N).
\end{equation*}
\end{conj}
Note that since the structure of $\scC(N)$ is completely determined by the third author \cite{Yoo}, the computation of $J_0(N)(\Q)_\tor$ is reduced to proving the conjectures above. (We only consider subgroups of the cuspidal subgroup of $J_0(N)$ here. 
Note that subgroups of the cuspidal subgroup of $J_1(N)$, the Jacobian variety of the modular curve $X_1(N)$, have been also studied by many mathematicians. For instance, see \cite{Ha98, Su10, Ta92, Ta93, Ta12, Ta14, Ya09, Yu80}.)

\ms
In this paper, we are primarily concerned with Conjecture \ref{conjecture}. (For Conjecture \ref{conjecture: GOC}, See \cite{Yoo1}.) To the best knowledge of the authors, Conjecture \ref{conjecture} is only known when $N$ is one of the following cases:
\begin{enumerate}
\item
$N$ is small enough.
\item
$N=2^r M$ with $0\leq r \leq 3$ and $M$ odd squarefree. 
\item
$N=n^2 M$ with $n \dd 24$ and $M$ squarefree. 
\end{enumerate}
The second case is obvious as all the cusps of $X_0(N)$ are defined over $\Q$, and so $\scC(N)=\scC_N(\Q)=\scC_N$.  The third case is the main result of \cite{WY20}. Our main theorem is the following.

\begin{thm}\label{theorem: main theorem}
Let $N=p^2M$ for a prime $p$ and a squarefree integer $M$. Then we have
\begin{equation*}
\scC_{N}(\Q)=\scC(N).
\end{equation*}
\end{thm}


Note that the case where $M$ is divisible by $p$ is permitted. Here we remark that the proof of Conjecture \ref{conjecture} in the
case $N=n^2M$ with $n \dd 24$ and $M$ squarefree relies crucially on the
facts that all modular units (i.e., modular functions with divisors
supported on cusps) on $X_0(N)$ can be expressed as products of eta
functions of the form $\eta(\tau+k/h)$ with $h\dd 24$, and one can completely
determine the necessary and sufficient conditions for such an
eta product to be a modular unit on $X_0(N)$. However, as pointed
out in \cite{WY20}, this construction and characterization of modular
units work only when the level $N$ is of the special form $N=n^2M$
with $n\dd 24$ and $M$ squarefree. (This is related to the fact that if
$n$ is a positive integer such that $a^2\equiv1\mod n$ for all $a$
with $(a,n)=1$, i.e., if $(\Z/n\Z)^\times$ is an elementary
$2$-group, then $n$ divides $24$.) For general levels, we will need a very
different method for constructing modular units. It turns out that it
is possible to construct a family of functions using the Siegel
functions such that every modular unit on $X_0(N)$ can be uniquely
expressed as a product of functions from the family (see Theorem
\ref{theorem: theorem 1} below). However, it appears rather difficult
to determine the necessary and sufficient conditions for a product of
these functions to be modular on $X_0(N)$ and we are able to overcome
the difficulty only for a special type of $N$
(see Theorem \ref{theorem: modular unit iff} below).
This is the primary reason why we restrict our attention to this
special form of $N$ in Theorem \ref{theorem: main theorem}. 

If $[D] \in \scC_N(\Q)$, i.e., $D$ is a cuspidal divisor of degree $0$ satisfying $\sigma(D)\sim D$ for all $\sigma\in\Gal(\overline\Q/\Q)$, then using Theorem \ref{theorem: modular unit iff} we can now construct a modular unit $f$ such that $D+\operatorname{div}f$ is rational and thereby establish Theorem \ref{theorem: main theorem}. For more details, see Section \ref{section4}.

Before proceeding, we fix some notation.

\begin{notation}\label{nota}
\begin{itemize}
\item
Let $N$ be a positive integer and let $p$ denote a prime (unless otherwise mentioned). For example, we write $\prod_{p \mid n}$ to mean that the product is over the prime divisors of $n$.
\item
For a positive divisor $m$ of $N$, let $\ell(m)$ be the largest integer whose square divides $N/m$. Also, let
\begin{equation*}
m':=\frac{N}{m},\quad m'':=\frac{m'}{\ell(m)}=\frac{N}{m\ell(m)} \qa N':=\frac{N}{\ell(m)}.
\end{equation*}
If there is no confusion, we simply denote $\ell(m)$ by $\ell$.

\item
Let $L=\ell(1)$, i.e., $L$ is the largest integer whose square divides $N$.

\item
Let $\mathcal{D}_N$ be the set of all positive divisors of $N$ different from $N$. 

\item
For a positive integer $n$, let $\z_n:=e^{2 \pi i/n}$. 

\item
For any integer $s$ prime to $L$, let $\s_s$ be the element of $\Gal(\Q(\z_L)/\Q)$ such that $\s_s(\z_L)=\z_L^s$.

\item
For a meromorphic function $F$ on a Riemann surface $X$, we denote by $\tn{div} (F)$ the divisor of a function $F$, namely,
\begin{equation*}
\tn{div}(F):=\sum_{P \in X}\tn{ord}_P(F) \cdot P
\end{equation*}
where $\tn{ord}_P(F)$ is the order of vanishing of $F$ at a point $P$ in $X$.
\end{itemize} 
\end{notation}

Recall that a cusp of $X_0(N)$ can be represented as $\vect a c$, where $c$ is a positive divisor of $N$ and $a$ is an integer prime to $N$.\footnote{In most literatures, one takes an integer $a$ satisfying $(a, c)=1$. However, since the natural map $(\zmod N)^\times \to (\zmod c)^\times$ is surjective, we can always take $b$ such that $(b, N)=1$ and $b \equiv a \pmod c$. Then two cusps $\vect a c$ and $\vect b c$ are equivalent.} For example, the cusps $\infty$ and $0$ are written as $\vect 1 N$ and $\vect 1 1$, respectively.
Such a cusp $\vect a c$ of $X_0(N)$ is called a \textit{cusp of level $c$}.
Two cusps $\vect a c$ and $\vect {a'}{c'}$ (with $1\leq c, c' \dd N$ and $(aa', N)=1$) are equivalent if and only if $c=c'$ and $a \equiv a' \pmod z$, where $z=(c, N/c)$. So the number of all cusps of level $c$ is $\p(z)$, where $\p$ is Euler's totient function, and the number of all cusps of $X_0(N)$ is $\sum_{1\leq c \mid N} \p(\gcd(c, N/c))$. 
A cusp of level $c$ is defined over $\Q(\z_z)$ and so all cusps of $X_0(N)$ are defined over $\Q(\z_L)$. Moreover, for any $\s_s \in \Gal(\Q(\z_L)/\Q)$ we have 
\begin{equation}\label{equation: cusp Galois action}
\s_s \vect a c = \vect {s^* a}{c},
\end{equation}
where $s^* \in \Z$ is chosen so that $ss^* \equiv 1 \pmod L$ and $(s^*, aN)=1$. For more details, see \cite[Sec. 2]{Yoo}.

\ms
Let $D$ be a cuspidal divisor of degree $0$ on $X_0(N)$ such that $[D] \in \scC_N(\Q)$. (Here, $[D]$ denotes the linear equivalence class of a divisor $D$.)
To prove Conjecture \ref{conjecture}, we must find a rational cuspidal divisor $D'$ which is linearly equivalent to $D$, i.e., $[D]=[D']$. In other words, it is necessary to find a modular function $F$ on $X_0(N)$ such that $D+\tn{div} (F)$ is a rational cuspidal divisor. 
Thus, it is desirable to have a nice description of such a function $F$, which is called a \textit{modular unit} on $X_0(N)$. Since a modular unit on $X_0(N)$ is also that on $X(N)$, it can be written as a product of Siegel functions (cf. \cite{KL}). In this paper, we use generalized Dedekind eta functions which are variants of Siegel functions.
\begin{defn}
For integers $g$ and $h$ not both congruent to $0$ modulo $N$, let
\begin{equation*}
E_{g, h}(\tau):=q^{B_2(g/N)/2} \prod_{n=1}^\infty \left(1-\z_N^h q^{n-1+g/N}\right)\left(1-\z_N^{-h}q^{n-g/N}\right),
\end{equation*}
where $q=e^{2\pi i \tau}$ and $B_2(x)=x^2-x+1/6$ is the second Bernoulli polynomial.
\end{defn}

Motivated by the previous work of Wang and the second author \cite{WY20}, we construct the following.

\begin{defn}\label{definition: Fmh}
For each $m\in \mathcal{D}_N$, we fix a set $S_{m''} \subset \{1, \dots, m''-1\}$ of representatives of $(\zmod {m''})^\times /\{\pm 1\}$. For each $\a \in S_{m''}$, let $\d \in \{ 1, \dots, m''-1 \}$ be an integer such that $\a \d \equiv 1 \pmod {m''}$. 
If $m''\neq 2$, we set
\begin{equation*}
F_{m, h}(\tau):=\prod_{\a \in S_{m''}} E_{\a m \ell, \d h N'}(N'\tau).
\end{equation*}
Also, if $m''=2$, we define
\begin{equation*}
F_{m, h}(\tau):=E_{m\ell, hN'}(N'\tau)^{1/2}.
\end{equation*}
\end{defn}

Using these functions, we can produce all modular units on $X_0(N)$.
\begin{thm}\label{theorem: theorem 1}
Every modular unit on $X_0(N)$ can be uniquely expressed as 
\begin{equation}\label{equation: 101}
\e\prod_{m \in \mathcal{D}_N} \prod_{h=0}^{\p(\ell(m))-1} F_{m, h}^{e_{m, h}} ~~~ \text{ for some } ~e_{m, h} \in \Z \qqa \e \in \C^\times.
\end{equation}
\end{thm}
A variant of the theorem is the following.

\begin{thm}\label{theorem: variant of theorem 1}
Every modular unit on $X_0(N)$ can be uniquely expressed as 
\begin{equation*}
\e\prod_{m \in \mathcal{D}_N}  \prod_{h=1}^{\p(\ell(m))} F_{m, h}^{e_{m, h}} ~~~ \text{ for some } ~e_{m, h} \in \Z  \qqa \e \in \C^\times.
\end{equation*}
\end{thm}

Unfortunately, these results do not answer when such products are indeed modular units on $X_0(N)$. A partial result is the following, which is a generalization of the well-known criteria of Ligozat \cite[Prop. 3.2.1]{Li75}. 
\begin{thm}\label{theorem: modular unit iff}
Suppose that $L$ is odd and squarefree. If $(L, \p(N/L))=1$, then\footnote{By Remark \ref{remark: range h}, this product is the most general one constructed from $F_{m, h}$.} 
\begin{equation*}
f=\prod_{m \in \mathcal{D}_N} \prod_{h=0}^{\ell(m)-1} F_{m, h}^{e_{m, h}} ~~\text{ for some } ~e_{m, h} \in \Z
\end{equation*}
is a modular function on $X_0(N)$ if and only if all the following conditions are satisfied:
\begin{enumerate}
\item
the order of $f$ at the cusp $\infty$ is an integer;
\item
the order of $f$ at the cusp $0$ is an integer;
\item 
the order of $f$ at the cusp $\vect 1 {N_0}$ is an integer, where $N_0$ is the odd part of $N$;
\item (the mod $L$ condition) we have
\begin{equation*}
\sum_{m \in \mathcal{D}_N} m \p(m'') \sum_{h=1}^{\ell(m)-1} h e_{m, h} \equiv 0 \pmod L;
\end{equation*}
\item (the mod $2$ condition)
for each odd prime divisor $p$ of $N$,  we have
\begin{equation*}
\sum_{m: m'' = p^r} \sum_{h=0}^{\ell(m)-1} e_{m, h} \equiv 0 \pmod 2,
\end{equation*}
where the first sum runs over all $m \in \mathcal{D}_N$ such that $m''$ is a power of $p$.
\end{enumerate}
\end{thm}

With all these results together (and some properties of the functions $F_{m, h}$ which are discussed in Section \ref{section2}), one may try to prove Conjecture \ref{conjecture}. One of the key observations is that the prime-to-$\p(L)$ parts of $\scC_N(\Q)$ and $\scC(N)$ coincide, which can be easily proved by an averaging method (cf. Lemma \ref{lemma: averaging case}). As a result, the following is enough for our purpose.
\begin{thm}\label{theorem: modular unit suff}
Let 
\begin{equation}\label{equation: 102}
f=\prod_{m \in \mathcal{D}_N} \prod_{h=0}^{\ell(m)-1} F_{m, h}^{e_{m, h}} ~~\text{ for some } ~e_{m, h} \in \Z.
\end{equation}
Then $f^L$ is a modular function on $X_0(N)$ if all the following conditions are satisfied:
\begin{enumerate}
\item
the order of $f$ at the cusp $\infty$ is an integer;
\item
the order of $f$ at the cusp $0$ is an integer; 
\item 
the order of $f$ at the cusp $\vect 1 {N_0}$ is an integer, where $N_0$ is the odd part of $N$;
\item
(the mod $2$ condition) for each odd prime divisor $p$ of $N$,
\begin{equation*}
\sum_{m: m'' = p^r} \sum_{h=0}^{\ell(m)-1} e_{m, h} \equiv 0 \pmod 2,
\end{equation*}
where the first sum runs over all $m \in \mathcal{D}_N$ such that $m''$ is a power of $p$.
\end{enumerate}
\end{thm}

Now we hope to prove Conjecture \ref{conjecture} under the assumption that $(L, \p(L))=1$.
Note that there is a difference between the ranges of $h$ in the expressions (\ref{equation: 101}) and (\ref{equation: 102}). Although this is a huge obstacle in our method, we were able to overcome this when $L$ is a prime. In the forthcoming work \cite{YY}, we develop a new method to overcome this obstacle when $L$ is divisible by at most two primes.

\ms
The organization of the paper is as follows. In Section \ref{section2}, we study various properties of the functions $F_{m, h}$ and prove Theorems \ref{theorem: theorem 1} and \ref{theorem: variant of theorem 1}. In Section \ref{section3}, we discuss some criteria for a product of the functions $F_{m, h}$ to be a modular function on $X_0(N)$ and prove Theorems \ref{theorem: modular unit iff} and \ref{theorem: modular unit suff}. Lastly, we prove Theorem \ref{theorem: main theorem} in Section \ref{section4}.

\ms
\section{Construction of modular units}\label{section2}
In the 1950s, Newman constructed modular units on $X_0(N)$ using the Dedekind eta function \cite{Ne57, Ne59}. Also, Ogg and Ligozat studied modular units on $X_0(N)$ whose divisors are rational \cite{Og73, Og74}. As a result, Ligozat could prove Conjectures \ref{conjecture: GOC} and \ref{conjecture} for the case where $J_0(N)$ is an elliptic curve over $\Q$ \cite{Li75}. Then Chua and Ling computed the cuspidal subgroup of $J_0(pq)$ for two primes $p$ and $q$ using the method of Ogg and Ligozat \cite{CL97}. More generally, for a squarefree integer $N$, Takagi proved that all modular units on $X_0(N)$ can be constructed using the Dedekind eta function and computed the size of the cuspidal subgroup of $J_0(N)$ \cite{Ta97}. Lastly, Wang and the second author proved that all modular units on $X_0(n^2M)$ with $n \dd 24$ and $M$ squarefree can be constructed using generalized Dedekind eta functions (which are different from $E_{g, h}$) \cite{WY20}.  

In this section, we prove that all modular units on $X_0(N)$ can be constructed using the functions $F_{m, h}$ (Theorems \ref{theorem: theorem 1} and \ref{theorem: variant of theorem 1}). To do so, we first study various properties of the functions $F_{m, h}$. Before proceeding, we recall the transformation properties of the function $E_{g, h}$.
\begin{prop}[{\cite[Th. 1]{Ya04}}]\label{proposition: 2.1}
The functions $E_{g, h}$ satisfy
\begin{equation}\label{equation: 1}
E_{g+N, h}=E_{-g, -h}=-\z_N^{-h}E_{g, h} \qa E_{g, h+N}=E_{g, h}.
\end{equation}
Moreover, let $\g=\mat a b c d \in \SL_2(\Z)$. Then for $c=0$, we have 
\begin{equation*}
E_{g, h}(\tau+b)=e^{\pi i b B(g/N)} E_{g, bg+h}(\tau),
\end{equation*}
and for $c \neq 0$,
\begin{equation*}
E_{g, h}(\g \tau)=\ve(a, b, c, d)e^{\pi i \d} E_{ag+ch, bg+dh}(\tau),
\end{equation*}
where
\begin{equation}\label{equation: 2}
\ve(a, b, c, d):=\begin{cases}
e^{\frac{\pi i}{6}(bd(1-c^2)+c(a+d-3))} & \text{ if $~c$ is odd}, \\
-ie^{\frac{\pi i}{6}(ac(1-d^2)+d(b-c+3))} & \text{ if $~d$ is odd}, \\
\end{cases}
\end{equation}
and 
\begin{equation*}
\d:=\frac{g^2 ab+2ghbc+h^2cd}{N^2}-\frac{gb+h(d-1)}{N}.
\end{equation*}
\end{prop}

\begin{rem}\label{remark: range h}
The function $F_{m, h}$ is dependent only on the residue class of $h$ modulo $\ell(m)$ as
 \begin{equation*}
E_{\a m \ell, \d (h+\ell) N'}=E_{\a m \ell, \d h N'+\d N}=E_{\a m \ell, \d h N'} \quad \tn{by (\ref{equation: 1})}.
\end{equation*}
\end{rem}

\begin{rem}
If $\a \equiv \a' \pmod {m''}$, then $(\a-\a')m\ell$ is divisible by $m''m\ell=mm'=N$. Also, if $\d \equiv \d' \pmod {m''}$, then $(\d-\d')N'$ is divisible by $m''N'=m'N/{\ell^2}$, which is a multiple of $N$. So by (\ref{equation: 1}), the functions $F_{m, h}$ defined using different choices of representatives for $(\zmod {m''})^\times/\{\pm 1\}$ differ only by roots of unity.
\end{rem}

Suppose that $m''=2$. Since $\ell$ divides $m''$, we have either $\ell=1$ or $\ell=2$. 
If $\ell=1$, then $m'=m''\ell=2$, $m=N/{m'}=N/2$ and $N'=N/\ell=N$. 
Note that 
\begin{equation*}
E_{N/2, 0}(\tau)=q^{-1/24}\prod_{n=1}^\infty (1-q^{n-\frac{1}{2}})^2=\frac{\eta(\tau/2)^2}{\eta(\tau)^2}.
\end{equation*}
Thus, we have 
\begin{equation}\label{equation: 3}
F_{N/2, 0}(\tau)=\frac{\eta((N/2)\tau)}{\eta(N\tau)}.
\end{equation}
If $\ell=2$, then $m'=m''\ell=4$, $m=N/{m'}=N/4$ and $N'=N/\ell=N/2$. Note that
\begin{equation*}
E_{N/2, N/2}(\tau)=q^{-1/24}\prod_{n=1}^\infty(1+q^{n-\frac{1}{2}})^2=\frac{\eta(\tau)^4}{\eta(\tau/2)^2 \eta(2\tau)^2}.
\end{equation*}
Thus, we have
\begin{equation}\label{equation: 4}
F_{N/4, 0}(\tau)=\frac{\eta((N/4)\tau)}{\eta((N/2)\tau)} \qa F_{N/4, 1}(\tau)=\frac{\eta((N/2)\tau)^2}{\eta((N/4)\tau)\eta(N\tau)}.
\end{equation}

So if $m''=2$, then for any $\g \in \G_0(N)$ we have $F_{m, h}(\g \tau)=\epsilon F_{m, h}(\tau)$ for some $24$-th root of unity $\epsilon$. Moreover, by the result of Ligozat \cite[Prop. 3.5]{Yoo} we can show that
$F_{m, h}^k$ is a modular function on $X_0(N)$ if and only if $k$ is divisible by $24$. 
In general, we have the following.

\begin{lem}\label{lemma: F is modular}
For any $m\in \mathcal{D}_N$ and $\g \in \Gamma_0(N)$, we have
\begin{equation*}
F_{m, h}(\g \tau)=\epsilon F_{m, h}(\tau)
\end{equation*}
for some $\lcm(2m'', 24)$-th root of unity $\epsilon$ depending on $m$, $h$ and $\g$. In particular,
\begin{equation*}
F_{m, h}(\tau)^{\lcm(2m'', 24)}
\end{equation*}
is a modular function on $X_0(N)$.
\end{lem}
\begin{proof}
By the discussion above, it suffices to prove the case where $m'' \neq 2$, which we assume from now on.

Let $\g=\mat a b c d \in \G_0(N)$. Note that $N'\g \tau=\mat a {bN'} {c/N'} {d} (N'\tau)$. 
So by Proposition \ref{proposition: 2.1}, if $c=0$, then 
\begin{equation*}
E_{\a m \ell, \d h N'}(N'\g \tau)=e^{\pi i bN' B_2(\a/{m''})} E_{\a m \ell, \d h N'}(N'\tau)
\end{equation*}
and hence 
\begin{equation*}
F_{m, h}(\g \tau)=\left(\prod_{\a \in S_{m''}} e^{\pi i bN' B_2(\a/{m''})} \right) F_{m, h}(\tau).
\end{equation*}
Since $m''$ divides $N'$, the result follows in this case. Again by Proposition \ref{proposition: 2.1}, if $c \neq 0$, then 
\begin{equation*}
\begin{split}
E_{\a m \ell, \d h N'}(N'\g \tau)&=\ve(a, bN', c/N', d) e^{2\pi i {A_\a}} E_{\a am\ell+\d ch, \a b m\ell N'+\d dh N'}(N'\tau)\\
&=\ve(a, bN', c/N', d) e^{2\pi i {A_\a}} (-\z_\ell^{-\d dh})^{\d ch/N} E_{\a am\ell, \d dhN'}(N'\tau),
\end{split}
\end{equation*}
where $\ve(a, bN', c/N', d)$ is a $24$-th root of unity determined by (\ref{equation: 2}) and 
\begin{equation*}
\begin{split}
A_\a&:=  \frac{(\a m \ell)^2 a b N'+2(\a m \ell) (\d h N') bc+(\d h N')^2 cd/N'}{2N^2}-\frac{(\a m\ell)bN'+\d h N'(d-1)}{2N}\\
&\equiv \frac{\a^2 mab}{2m''}+\frac{\d^2 h^2 cd/N}{2\ell}-\frac{\a mb}{2} -\frac{\d h (d-1)}{2\ell} \pmod \Z
\end{split}
\end{equation*}
so that $e^{2 \pi i {A_\a}}$ is a $2m''$-th root of unity. 
As $\a$ goes through $S_{m''}$, $\a a$ also goes through a set of representatives of $(\zmod {m''})^\times/{\{\pm 1\}}$. If $\a'$ is the element in $S_{m''}$ such that $\a a \equiv \pm \a' \pmod {m''}$, then $E_{\a a m\ell, \d dhN'}$ and $E_{\a' m \ell, \d' d h N'}$ differ by a $2\ell$-th root of unity. Therefore we have
\begin{equation*}
F_{m, h}(\g \tau)=\epsilon F_{m, h}(\tau)
\end{equation*}
for some $\lcm (2m'', 24)$-th root of unity $\epsilon$. This completes the proof.
\end{proof}

\begin{rem}\label{remark: root of unity for m''=2}
Let $\g=\mat a b c d \in \G_0(N)$ with $24 \dd b$, $24N \dd c$ and $c>0$, and let $\g'=\mat a {bM} {c/M} {d}$ for a divisor $M$ of $N$.
Since $d$ is odd, by \cite[Lem. 1]{Ya04} we have
\begin{equation*}
\eta(M \g\tau)=\eta(\g' (M\tau))=\leg{c/M}{d} e^{\frac{\pi i d}{4}}\sqrt{-i(c\tau+d)} ~\eta(M\tau),
\end{equation*}
where $\leg {\cdot}{\cdot}$ is the Legendre--Jacobi symbol.
Hence by (\ref{equation: 3}) and (\ref{equation: 4}), we have
\begin{equation*}
F_{N/2, 0}(\g \tau)=\leg{2}{d}F_{N/2, 0}(\tau)=(-1)^{\frac{d^2-1}{8}} F_{N/2, 0}(\tau),
\end{equation*}
\begin{equation*}
F_{N/4, 0}(\g \tau)=(-1)^{\frac{d^2-1}{8}} F_{N/4, 0}(\tau) \qa F_{N/4, 1}(\g\tau)=F_{N/4, 1}(\tau).
\end{equation*}
\end{rem}

\begin{lem}\label{lemma: order of Fmh}
Let $m \in \mathcal{D}_N$. Then for an integer $h$, the order of $F_{m, h}$ at a cusp $\vect a c$ of $X_0(N)$ is 
\begin{equation*}
\frac{\ell(N', c)^2}{4(c^2, N)} \sum_{\a \in (\zmod {m''})^\times} P_2\left(	\frac{\a a'}{m''}+\frac{\d hc'}{\ell}\right),
\end{equation*}
where $P_2(x)=B_2(\{x\})$ is the second Bernoulli function, $a'=\frac{N'a}{(N', c)}$ and $c'=\frac{c}{(N', c)}$. 
\end{lem}
\begin{proof}
Suppose that $m''=2$. For a divisor $M$ of $N$, the order of $\eta(M\tau)$ at the cusp $\vect a c$ is $\frac{N(c, M)^2}{24M(c^2, N)}$ by \cite[Prop. 3.2.8]{Li75} and so the result follows using (\ref{equation: 3}) and (\ref{equation: 4}).  

Next, suppose that $m'' \neq 2$. Let $b, d, b'$ and $d'$ be integers such that 
\begin{equation*}
\g=\mat a b c d \in \SL(2, \Z) \qa \g'=\mat {a'}{b'}{c'}{d'} \in \SL(2, \Z).
\end{equation*}
We check that
\begin{equation*}
N'\g \tau=\g' \left(\frac{c(c\tau+d)}{N'(c')^2}-\frac{d'}{c'}\right).
\end{equation*}
Hence we have
\begin{equation*}
\begin{split}
F_{m, h}(\g \tau)&=\prod_{\a \in S_{m''}} E_{\a m \ell, \d h N'}\left( \g' \left(\frac{c(c\tau+d)}{N'(c')^2}-\frac{d'}{c'}\right)\right)\\
&=\epsilon \prod_{\a \in S_{m''}} E_{\a m \ell a'+\d h N' c', \a m \ell b'+\d h N' d'} \left(\frac{c(c\tau+d)}{N'(c')^2}-\frac{d'}{c'}\right)
\end{split}
\end{equation*}
for some root of unity $\epsilon$. The Fourier expansion of $F_{m, h}(\g\tau)$ starts from $q^A$, where
\begin{equation*}
A=\frac{c^2}{2N'(c')^2} \sum_{\a \in S_{m''}} P_2\left( \frac{\a m \ell a'+\d h N' c'}{N}\right)=\frac{(N', c)^2}{2N'}\sum_{\a \in S_{m''}} P_2\left(\frac{\a a'}{m''} + \frac{\d h c'}{\ell}\right).
\end{equation*}
Since the width of a cusp $\vect a c$ is $\frac{N}{(c^2, N)}$ (cf. \cite[Lem. 2.10]{Yoo}), the order of $F_{m, h}$ at a cusp $\vect a c$ is 
\begin{equation*}
\frac{NA}{(c^2, N)}=\frac{\ell(N', c)^2}{4(c^2, N)}\sum_{\a \in (\zmod {m''})^\times} P_2\left(\frac{\a a'}{m''} + \frac{\d h c'}{\ell}\right)
\end{equation*}
as claimed. (Note that $P_2(x)$ is an even function since $P_2(-x)=P_2(1-x)=P_2(x)$.)
\end{proof}

\begin{cor}\label{corollary: order at special cusps}
The order of $F_{m, h}$ at the cusp $\infty$ is 
\begin{equation*}
\frac{m}{24}\prod_{p \mid m''}(1-p).
\end{equation*}
Also, the order of $F_{m, h}$ at the cusp $0$ is
\begin{equation*}
\frac{m''}{24\ell} \sum_{k \mid m''} \frac{\mu(k)}{k} (\ell, kh)^2,
\end{equation*}
where $\mu(k)$ is the M\"obius function.
Furthermore, if $N=2N_0$ for some odd integer $N_0$, then the order of $F_{m, h}$ at the cusp $\vect 1 {N_0}$ is
\begin{equation*}
\frac{m}{24(m, 2)} \prod_{p \mid m'', ~p\neq 2}(1-p).
\end{equation*}
\end{cor}
\begin{proof}
For the cusp $\infty$, we have $a=a'=1$, $c=N$ and $c'=N/{N'}=\ell$. So by Lemma \ref{lemma: sum of units mod m} below, we have
\begin{equation*}
\frac{\ell (N')^2}{4N} \sum_{\a \in (\zmod {m''})^\times} P_2\left(\frac{\a }{m''}\right)=\frac{N'}{24m''}\prod_{p \mid m''}(1-p)=\frac{m}{24}\prod_{p \mid m''}(1-p).
\end{equation*}

For the cusp $0$, we have $a=c=c'=1$ and $a'=N'$. Thus by Lemma \ref{lemma: sum of units mod m}, the order of $F_{m, h}$ at $0$ is
\begin{equation*}
\frac{\ell}{4}\sum_{\a \in (\zmod {m''})^\times} P_2\left(\frac{\d h}{\ell}\right)=\frac{\ell}{4}\sum_{\a \in (\zmod {m''})^\times} P_2\left(\frac{\a h}{\ell}\right)=\frac{m''}{24\ell} \sum_{k \mid m''} \frac{\mu(k)}{k}(\ell, kh)^2.
\end{equation*}

Lastly, let $N=2N_0$ for some odd integer $N_0$. Since $\ell$ is a divisor of $N_0$, we have $N'=2N''$, where $N''=N_0/\ell$. Thus for the cusp $\vect 1 {N_0}$, we have $(N', c)=N''$, $a'=2$ and $c'=\ell$. Hence by Lemma \ref{lemma: sum of units mod m}, the order of $F_{m, h}$ at the cusp $\vect 1 {N_0}$ is 
\begin{equation*}
\frac{N''}{4}\sum_{\a \in (\zmod {m''})^\times} P_2\left( \frac{2\a}{m''} \right)=\frac{N''}{24m''} \sum_{k \mid m''} \frac{\mu(k)}{k}(m'', 2k)^2=\frac{m}{48} \sum_{k \mid m''} \mu(k)k(m''/k, 2)^2.
\end{equation*}
If $m$ is odd, then $m''$ is even. Let $n=m''/2$. Then we have
\begin{equation*}
 \sum_{k \mid m''} \mu(k)k(m''/k, 2)^2=\sum_{k \mid n} \mu(k)k(2^2-2\cdot 1^2)=2\sum_{k \mid n} \mu(k) k=2\prod_{p \mid n}(1-p).
\end{equation*}
If $m$ is even, then $m''$ is odd. Thus, $(m''/k, 2)=1$ and so 
\begin{equation*}
 \sum_{k \mid m''} \mu(k)k(m''/k, 2)^2= \sum_{k \mid m''} \mu(k)k=\prod_{p \mid m''}(1-p).
\end{equation*}
This completes the proof.
\end{proof}

\begin{lem}\label{lemma: sum of units mod m}
Let $x$ be a positive integer. Then we have
\begin{equation*}
\sum_{\a \in (\zmod x)^\times} P_2\left(\frac{\a}{x}\right)=\frac{1}{6x}\prod_{p\mid x}(1-p).
\end{equation*}
More generally, let $y$ be a divisor of $x$. Then for any integer $n$, we have
\begin{equation*}
\sum_{\a \in (\zmod x)^\times} P_2\left(\frac{\a n}{y}\right)=\frac{x}{6y^2} \sum_{k\mid x} \frac{\mu(k)}{k}(y, kn)^2.
\end{equation*}
\end{lem}
\begin{proof}
By the M\"obius inversion formula, we have
\begin{equation*}
\sum_{\a \in (\zmod x)^\times} P_2\left(\frac{\a}{x}\right)=\sum_{k \mid x} \mu(k) \sum_{j=0}^{x/k-1} P_2\left(\frac{jk}{x}\right).
\end{equation*}

Recall that the second Bernoulli polynomial $B_2(x)$ satisfies
\begin{equation*}
\sum_{j=0}^{n-1} B_2\left(x +\frac{j}{n}\right)=\frac{1}{n} B_2(nx),
\end{equation*}
which implies
\begin{equation*}
\sum_{j=0}^{n-1} P_2\left(x +\frac{j}{n}\right)=\frac{1}{n} P_2(nx).
\end{equation*}
Thus, we have
\begin{equation*}
\sum_{k\mid x} \mu(k) \sum_{j=0}^{x/k-1} P_2\left(\frac{j}{x/k}\right)=\sum_{k \mid x} \mu(k)\left(\frac{k}{x} P_2(0)\right)=\frac{1}{6x}\sum_{k \mid x} \mu(k)k=\frac{1}{6x}\prod_{p \mid x}(1-p).
\end{equation*}

Next, for each divisor $k$ of $x$, let 
\begin{equation*}
z=\frac{y}{(y, kn)} \qa k'=\frac{kn}{(y, kn)}.
\end{equation*}
Note that by definition, $(z, k')=1$. Again by the M\"obius inversion formula, we have
\begin{equation*}
\begin{split}
\sum_{\a \in (\zmod x)^\times} P_2\left(\frac{\a n}{y}\right)&=\sum_{k\mid x} \mu(k) \sum_{j=0}^{x/k-1} P_2\left(\frac{jkn}{y}\right)\\
&=\sum_{k\mid x} \mu(k) \sum_{t=0}^{x/{(kz)}-1} \sum_{j=0}^{z-1} P_2\left(\frac{(j+z t)k'}{z}\right)\\
&=\sum_{k\mid x} \mu(k) \sum_{t=0}^{x/{(kz)}-1} \sum_{j=0}^{z-1} P_2\left(\frac{j}{z}\right)\\
&=\frac{x}{6} \sum_{k\mid x} \frac{\mu(k)}{kz^2}=\frac{x}{6y^2} \sum_{k\mid x} \frac{\mu(k)}{k}(y, kn)^2
\end{split}
\end{equation*}
as desired.
\end{proof}

\begin{lem}\label{lemma: galois action}
For an integer $s$ prime to $L$, we have
\begin{equation*}
\s_s(\tn{div} (F_{m, h})) = \tn{div} (F_{m, sh}).
\end{equation*}
\end{lem}
\begin{proof}
Let $\vect a c$ be a cusp of $X_0(N)$. By (\ref{equation: cusp Galois action}), it suffices to show that the order of $F_{m, h}$ at $\vect a c$ is equal to that of $F_{m, sh}$ at $\vect {s^* a}{c}$, where $s^* \in \Z$ such that $ss^* \equiv 1 \pmod L$ and $(s^*, aN)=1$.
 By Lemma \ref{lemma: order of Fmh}, the order of $F_{m, h}$ at $\vect a c$ is 
\begin{equation*}
\frac{\ell(N', c)^2}{4(c^2, N)} \sum_{\a \in (\zmod {m''})^\times} P_2\left(	\frac{\a a'}{m''}+\frac{\d hc'}{\ell}\right).
\end{equation*}

Note that $(s^*a)'=s^* a'$ and $c'$ is unchanged. Again by Lemma \ref{lemma: order of Fmh}, the order of $F_{m, sh}$ at a cusp $\vect {s^*a} c$ is
 \begin{equation*}
\frac{\ell(N', c)^2}{4(c^2, N)} \sum_{\a \in (\zmod {m''})^\times} P_2\left(	\frac{\a s^*a'}{m''}+\frac{\d shc'}{\ell}\right).
 \end{equation*}
 Since $s^*$ is relatively prime to $m''$, the multiplication by $s^*$ (and taking reduction modulo ${m''}$) defines a bijection on $(\zmod {m''})^\times$. Since $\a \d \equiv 1 \pmod {m''}$ and 
$s^* s \equiv 1 \pmod {m''}$, we have $(\a s^*)(\d s) \equiv 1 \pmod {m''}$. Thus, the result follows  by replacing $\a s^*$ and $\d s$ by $\a$ and $\d$, respectively.
\end{proof}

\begin{cor}\label{corollary: rationality}
Let $m\in \mathcal{D}_N$. For any integer $s$ prime to $L$, we have 
\begin{equation*}
\s_s(\tn{div}(F_{m, 0}))=\tn{div}(F_{m, 0}).
\end{equation*}
\end{cor}
\begin{proof}
The result is obvious by Lemma \ref{lemma: galois action}.
\end{proof}

\begin{rem}\label{remark: order of Fm0}
By Lemmas \ref{lemma: order of Fmh} and \ref{lemma: sum of units mod m}, the order of $F_{m, 0}$ at a cusp $\vect a c$ is
\begin{equation*}
\frac{\ell(N', c)^2}{4(c^2, N)} \sum_{\a \in (\zmod {m''})^\times} P_2\left(	\frac{\a a'}{m''}\right)=\frac{\ell(N', c)^2}{24m''(c^2, N)}  \sum_{k \mid m''} \mu(k)k(m''/k, a')^2.
\end{equation*}
Since $a$ is relatively prime to $m''$, we have $(m''/k, a')=(m''/k, N'')$, where $N''=N'/{(N', c)}$. Thus, the order of $F_{m, 0}$ at a cusp $\vect a c$ is independent of $a$ and this gives another proof of Corollary \ref{corollary: rationality}. 
\end{rem}

\begin{rem}\label{example: 6}
By \cite[Th. 3.6]{Yoo}, $F_{m, 0}$ must be written as a product of Dedekind eta functions (up to constant). Indeed, we can explicitly write $F_{m, 0}$ as follows: If $m''\neq 2$, then
\begin{equation*}
\begin{split}
F_{m, 0}(\tau)&=\prod_{\a \in S_{m''}} E_{\a m \ell, 0}(N'\tau)\\
&=\e q^A \prod_{\a \in S_{m''}} \prod_{n=1}^\infty (1-q^{N'(n-1)+\a m})(1-q^{N'(n-1)+(m''-\a)m})\\
&=\e q^A \prod_{\a \in (\zmod {m''})^\times} \prod_{n=0}^\infty (1-q^{m(m''n+\a)})=\e q^A \prod_{n=1, (n, m'')=1}^\infty (1-q^{mn})\\
&=\e q^A \prod_{k \mid m''} \left(\prod_{n=1}^\infty (1-q^{kmn})\right)^{\mu(k)}=\e \prod_{k \mid m''} \eta(km\tau)^{\mu(k)},
\end{split}
\end{equation*}
where $\epsilon$ is a root of unity and $A$ is the order of $F_{m, 0}$ at the cusp $\infty$. Note that by (\ref{equation: 3}) and (\ref{equation: 4}), the same holds for $m''=2$.

Alternatively, by \cite[Prop. 3.2.8]{Li75} the order of $\eta(km\tau)$ at a cusp $\vect a c$ of $X_0(N)$ is $\frac{N(c, km)^2}{24km(c^2, N)}$. Thus, the order of
$\prod_{k \mid m''} \eta(km\tau)^{\mu(k)}$ at a cusp $\vect a c$ is
\begin{equation*}
\begin{split}
&\phantom{=}\frac{m'}{24(c^2, N)}\sum_{k \mid m''} \frac{\mu(k)}{k}(c, km)^2=\frac{\ell}{24(c^2, N)}\sum_{k'\mid m''} \mu(m''/k')k'(c, N'/k')^2\\
&=\frac{\ell}{24(c^2, N)}\sum_{k' \mid m''} \frac{\mu(m''/k')}{k'}(ck', N')^2=\frac{\ell(N', c)^2}{24(c^2, N)} \sum_{k' \mid m''}\frac{\mu(m''/k')}{k'}(k', N'')^2\\
&=\frac{\ell(N', c)^2}{24m''(c^2, N)}\sum_{k \mid m''} \mu(k)k(m''/k, N'')^2,
\end{split}
\end{equation*}
where $k'=m''/k$ and $N''=N'/{(N', c)}$. Therefore by Remark \ref{remark: order of Fm0}, 
there is a constant $\e \in \C^\times$ such that
\begin{equation*}
F_{m, 0}=\e \prod_{k \mid m''} \eta(km\tau)^{\mu(k)}.
\end{equation*}
(Thus, the formula in Lemma \ref{lemma: order of Fmh} agrees with the formulas in the previous literatures.)
\end{rem}

\begin{rem}\label{example: 7}
Let $m \in \mathcal{D}_N$ such that $\ell(m) \dd 24$. Then for all $\a \in S_{m''}$, we have $\a \equiv \d \pmod \ell$. Thus, 
\begin{equation*}
\z_N^{\d h N'}=\z_\ell^{\d h}=\z_\ell^{\a h}
\end{equation*}
and we may write $E_{\a m \ell, \d h N'}(N'\tau)$ as 
\begin{equation*}
\begin{split}
E_{\a m \ell, \d h N'}(N'\tau)&=q^{N'B_2(\a/{m''})/2} \prod_{n=1}^{\infty}(1-e^{2\pi i (N'(n-1)\tau+\a m\tau+\a h/\ell)})(1-e^{2\pi i (N' n \tau-\a m\tau-\a h /\ell)})\\
&=q^{N'B_2(\a/{m''})/2} \prod_{n=1}^{\infty}(1-e^{2\pi i(m''(n-1)+\a)(m\tau+h/\ell)})(1-e^{2\pi i(m''n-\a)(m\tau+h/\ell)}).
\end{split}
\end{equation*}
It follows that
\begin{equation*}
F_{m, h}(\tau)=\e q^A \prod_{n=1, (n, m'')=1}^\infty (1-e^{2\pi i n(m\tau+h/\ell)})=\e' \prod_{k \mid m''} \eta(k(m\tau+h/\ell))^{\mu(k)}
\end{equation*}
for some roots of unity $\e$ and $\e'$, and $A$ is the order of $F_{m, h}$ at the cusp $\infty$. Hence the functions $F_{m, h}$ are products of the functions $\eta(k(m\tau+h/\ell))$, which are used in \cite{WY20} to construct modular units on $X_0(n^2 M)$ with $n \dd 24$ and $M$ squarefree.
\end{rem}

\begin{lem}\label{lemma: lemma 9}
Let $m\in \mathcal{D}_N$ such that $m''\neq 2$. Suppose that $\g=\mat a b c d \in \SL(2, \Z)$ such that $24 \dd b$, $24N \dd c$ and $c\neq 0$. Then we have
\begin{equation*}
F_{m, h}(\g \tau)=(-1)^{\frac{(d-1)\p(m'')}{4}} e^{2\pi i B} \prod_{\a \in S_{m''}} E_{\a am\ell, \d dhN'} (N'\tau),
\end{equation*}
where
\begin{equation*}
B=\frac{h(1-d)}{2\ell} \sum_{\a \in S_{m''}} \d.
\end{equation*}
\end{lem}
\begin{proof}
Since $c\neq 0$, by the computation in the proof of Lemma \ref{lemma: F is modular} we have
\begin{equation*}
E_{\a m \ell, \d h N'}(N'\g \tau)=\ve(a, bN', c/{N'}, d) e^{2 \pi i A_\a} (-\z_\ell^{-\d dh})^{\d c h/N} E_{\a am\ell, \d dhN'} (N'\tau).
\end{equation*}
Since $24 \dd b$ and $24N \dd c$, by (\ref{equation: 2}) we have $\ve(a, bN', c/{N'}, d)=-ie^{\pi i d/2}=(-1)^{\frac{d-1}{2}}$.

For the terms involving $A_\a$, we compute that
\begin{equation*}
\begin{split}
\sum_{\a \in S_{m''}} (\a^2+(m''-\a)^2)&=\sum_{\a \in (\zmod {m''})^\times} \a^2=\sum_{k \mid m''} \mu(k)k^2 \sum_{j=0}^{m''/k-1} j^2\\
&=\frac{1}{6}\sum_{k \mid m''} \mu(k)k^2 \left(\frac{m''}{k}-1\right) \frac{m''}{k} \left(\frac{2m''}{k}-1\right)\\
&=\frac{(m'')^2 \p(m'')}{3}+\frac{m''}{6}\prod_{p\mid m''} (1-p)
\end{split}
\end{equation*}
and consequently
\begin{equation*}
\sum_{\a \in S_{m''}} \a^2=m''\sum_{\a \in S_{m''}} \a-\frac{(m'')^2 \p(m'')}{12}+\frac{m''}{12}\prod_{p \mid m''}(1-p).
\end{equation*}
The same equality holds when we replace $\a$ by $\d$. Thus, in particular, we have
\begin{equation*}
\frac{12}{m''} \sum_{\a \in S_{m''}} \a^2  \equiv \frac{12}{m''} \sum_{\a \in S_{m''}} \d^2 \equiv 0 \pmod \Z.
\end{equation*}
Therefore, since $24 \dd b$ and $24N \dd c$, the following are all integers:
\begin{equation*}
\sum_{\a \in S_{m''}} \frac{\a^2 m ab}{2m''}, \quad \sum_{\a \in S_{m''}} \frac{\d^2 dh^2 c}{2\ell N}, \quad \sum_{\a \in S_{m''}} \frac{\a m b}{2},
\end{equation*}
and hence
\begin{equation*}
\sum_{\a \in S_{m''}} A_\a \equiv \frac{h(1-d)}{2\ell} \sum_{\a \in S_{m''}} \d \pmod \Z.
\end{equation*}
Similarly, we have $\sum_{\a \in S_{m''}} \frac{\d^2 dh^2c}{\ell N}\in \Z$ and hence $\prod_{\a \in S_{m''}} (-\z_\ell^{-\d dh})^{\d ch/N}=1$. This completes the proof.
\end{proof}

\ms
We finish this section by proving Theorems \ref{theorem: theorem 1} and \ref{theorem: variant of theorem 1}.
\begin{proof}[Proof of Theorem \ref{theorem: theorem 1}]
By the same argument as in the proof of \cite[Th. 2]{WY20}, we can deduce the result. Indeed, let
\begin{equation*}
\scU(N):=\{\tn{modular units on } X_0(N)\}/{\C^\times}.
\end{equation*}
and
\begin{equation*}
\scS(N):=\bigcup_{m \in \mathcal{D}_N} \{(m, h) : 0\leq h \leq \p(\ell(m))-1\}.
\end{equation*}
We shall prove the following four claims. 
\begin{enumerate}[(a)]
\item \label{item: a}
The cardinality of the set $\scS(N)$ is equal to the rank of $\scU(N)$, i.e., the number of the cusps of $X_0(N)$ minus one.
\item \label{item: b}
There are no multiplicative relations among $F_{m, h}$ for $(m, h) \in \scS(N)$.
\item \label{item: c}
Let $\scU_0$ be the subgroup of $\scU(N)$ formed by the products of $F_{m, h}$ for $(m, h) \in \scS(N)$. Then $\scU_0$ has the same rank as $\scU(N)$, i.e., $\scU_0$ is of finite index in $\scU(N)$.
\item \label{item: d}
If $g \in \scU(N)$ and $k$ is a positive integer such that $g^k \in \scU_0$, then $g \in \scU_0$.
\end{enumerate}

To prove Claim \ref{item: a}, we first define a bijection $\iota$ on the set of all positive divisors of $N$ as follows:
For $r \in \Z$, let $\iota_r(f):=\gauss{\frac{r+f+1}{2}}$ (resp. $\gauss{\frac{r-f}{2}}$) if $f$ is even (resp. $f$ is odd). Then for a positive integer $r$, $\iota_r$ induces a bijection on the set $\{0, \dots, r\}$. Note that for any $0\leq f \leq r$, we have ${\gauss{(r-f)/2}}=\min(\iota_r(f), r-\iota_r(f))$. Thus, we have
\begin{equation}\label{equation}
p^{\gauss{(r-f)/2}}=(p^{\iota_r(f)}, p^{r-\iota_r(f)}) \quad \text{ for any primes } ~p.
\end{equation}
Now, let $N=\prod_{i=1}^t p_i^{r_i}$ be the prime factorization of $N$. For a divisor $m=\prod_{i=1}^t p_i^{f_i}$ of $N$, let
\begin{equation*}
\iota(m):=\prod_{i=1}^t p_i^{\iota_{r_i}(f_i)}.
\end{equation*}
Then the map $\iota$ yields a bijection on the set of all positive divisors of $N$ as claimed. Furthermore, by (\ref{equation}) we have
\begin{equation*}
\ell(m)=\prod_{i=1}^t p_i^{\gauss{(r_i-f_i)/2}} = \prod_{i=1}^t (p_i^{\iota_{r_i}(f_i)}, p_i^{r_i-\iota_{r_i}(f_i)}) =(\iota(m), N/{\iota(m)}).
\end{equation*}
Hence for any positive divisor $m$ of $N$, the cardinality of the set $\{ (m, h) : 0\leq h \leq \p(\ell(m))-1\}$ is equal to the number of cusps of level $\iota(m)$.
Since $\iota(N)=\prod_{i=1}^t p_i^{\iota_{r_i}(r_i)}=\prod_{i=1, r_i \in 2\Z}^t p_i^{r_i}$, the number of cusps of level $\iota(N)$ is $1$. 
Thus, Claim \ref{item: a} follows.

To prove Claim \ref{item: b}, we observe that for $m''\neq 2$, the Fourier expansion of $F_{m, h}$ is of the form:
\begin{equation*}
\begin{split}
&\prod_{\a \in S_{m''}} q^{B_2(\a/{m''})/2} \prod_{n=1}^\infty (1-\z_{\ell(m)}^{\d h} q^{N'(n-1)+\a m})(1-\z_{\ell(m)}^{-\d h} q^{N' n-\a m})\\
&=\e q^A\prod_{\a \in S_{m''}} \prod_{n=1}^{\infty} (1-\z_{\ell(m)}^{\d h} q^{N'(n-1)+\a m})(1-\z_{\ell(m)}^{-\d h} q^{N'(n-1) +(m''-\a)m})\\
&=\e q^A \prod_{\a=1, (\a, m'')=1}^{m''-1} \prod_{n=0}^{\infty} (1-\z_{\ell(m)}^{\d h} q^{N'n+\a m})=\e q^A(1-\z_{\ell(m)}^h q^m+\cdots).
\end{split}
\end{equation*}
(If $m''=2$, then the Fourier expansion of $F_{m, h}$ is defined over $\Q$.)
Assume that $\prod_{(m, h) \in \scS(N)} F_{m, h}^{e_{m, h}}$ is a constant function. Considering the second term of the 
Fourier expansion, we find
\begin{equation*}
\sum_{h=0}^{\p(\ell(1))-1} e_{1, h} \z_{\ell(1)}^h=0.
\end{equation*}
Since $1, \z_{\ell(1)}, \dots, \z_{\ell(1)}^{\p(\ell(1))-1}$ form a basis for $\Q(\z_{\ell(1)})$ over $\Q$, we have $e_{1, h}=0$ for all $h$. Similarly, for the next divisor $m$ of $N$, we easily prove that $e_{m, h}=0$ for all $h$ by considering the Fourier coefficient of $q^{m+A}$. Continuing this way, we conclude that $e_{m, h}=0$ for all $(m, h) \in \scS(N)$. 

Claim \ref{item: c} follows from Lemma \ref{lemma: F is modular}, which states that for each $F_{m, h}$, there exists a positive integer $k$ such that $F_{m, h}^k \in \scU_0$. Finally, to prove Claim \ref{item: d} we follow the argument as in the proof of \cite[Th. 2]{WY20}. The key properties used in the proof are that $1, \z_{\ell(m)}, \dots, \z_{\ell(m)}^{\p(\ell(m))-1}$ form an integral basis for $\Q(\z_{\ell(m)})$ over $\Q$ and that the second Fourier coefficient of $F_{m, h}$ is $\z_{\ell(m)}^h$ (after normalizing the first Fourier coefficient to $1$) seen as above. This completes the proof of the theorem.
\end{proof}

\begin{proof}[Proof of Theorem \ref{theorem: variant of theorem 1}]
The proof is almost the same as that of Theorem \ref{theorem: theorem 1}. The fact we use is that $\z_{\ell(m)}$, $\dots$, $\z_{\ell(m)}^{\p(\ell(m))}$ form an integral basis for $\Q(\z_{\ell(m)})$ over $\Q$.
\end{proof}

\ms
\section{Criteria for modular units}\label{section3}
In this section, we prove Theorems \ref{theorem: modular unit iff} and \ref{theorem: modular unit suff}.
Before proceeding, we verify that Theorem \ref{theorem: modular unit iff} is a direct generalization of the well-known criteria of Ligozat \cite[Prop. 3.2.1]{Li75}. More specifically,
let $N$ be an odd squarefree integer. Then $L=1$ and $m''$ is a prime power $p^k$ if and only if $m=N/p$ and $m'=m''=p$. Thus, the conditions in Theorem \ref{theorem: modular unit iff} reduce to
\begin{enumerate}
\item
the order of $f$ at $\infty$ is an integer,
\item
the order of $f$ at $0$ is an integer, and
\item
for any prime divisor $p$ of $N$, $e_{N/p, 0}$ is even.
\end{enumerate}
According to Remark \ref{example: 6}, we have
\begin{equation*}
F_{m, 0}(\tau)=\e\prod_{k \mid m''} \eta(km\tau)^{\mu(k)}
\end{equation*}
for some root of unity $\e$. Note that
\begin{equation*}
\prod_{k\mid m''} (km)^{\mu(k)} \in p^a (\Q^\times)^2,
\end{equation*}
where $a=1$ if $m''=p^r$ is a prime power, and $a=0$ if $m''$ has at least two distinct prime divisors. Thus, the conditions above are equivalent to saying that $\prod_{\d \mid N} \eta(\d \tau)^{r_{\d}}$ (with $r_{\d} \in \Z$) is a modular function on $X_0(N)$ if and only if
\begin{enumerate}
\item[(0)]
$\sum r_\d =0$,
\item
$\sum r_\d \d \equiv 0 \pmod {24}$,
\item
$\sum r_\d (N/\d) \equiv 0 \pmod {24}$,
\item
$\prod \d^{r_\d} \in (\Q^\times)^2$.
\end{enumerate}

\ms
As an application of Theorem \ref{theorem: modular unit suff}, we obtain the following.
\begin{cor}\label{corollary: modular function Gmh}
Let $n=(3, L)$. If $L$ is odd, then
\begin{equation*}
\prod_{m \in \mathcal{D}_N} \prod_{h=1}^{\ell(m)-1} \left(\frac{F_{m, h}}{F_{m, 0}}\right)^{nL a_{m, h}} ~~ \text{ for some } ~a_{m, h} \in \Z
\end{equation*}
is a modular function on $X_0(N)$.
\end{cor}
\begin{proof}
Let $m\in \mathcal{D}_N$. Suppose that $\ell(m)>1$ and let $G_{m, h}:=\left(\frac{F_{m, h}}{F_{m, 0}}\right)^n$. It suffices to show that $G_{m, h}^L$ is a modular function on $X_0(N)$.
By Theorem \ref{theorem: modular unit suff}, it suffices to check that all the four conditions for $G_{m, h}$ are satisfied. By Corollary \ref{corollary: order at special cusps}, the orders of $G_{m, h}$ at the cusps $\infty$ and $\vect 1 {N_0}$ are $0$, so the first  and third conditions are satisfied. Also, the order of $G_{m, h}$ at $0$ is
\begin{equation*}
A=\frac{nm''}{24\ell} \sum_{k \mid m''}  \frac{\mu(k)}{k}((\ell, kh)^2-\ell^2).
\end{equation*}
For simplicity, let $d=(\ell, k)$ (which depends on $k$), $k=dk'$ and $\ell=d\ell'$. 
Since $(k', \ell')=1$, we have $(\ell, kh)=d(\ell', k'h)=d(\ell', h)$. Also, $a:=\frac{m''}{k'\ell'} \in \Z$ as $k'$ and $\ell'$ both divide $m''$ and $(k', \ell')=1$. Thus, we have 
\begin{equation*}
A = \frac{n}{24}\sum_{k \mid m''} \frac{m''d^2 \mu(k)}{k\ell}((\ell', h)^2-\ell'^2)=\frac{n}{24}\sum_{k \mid m''} \mu(k)a((\ell', h)^2-\ell'^2).
\end{equation*}
Note that since $\ell'$ is odd (we assume that $L$ is odd), we have $(\ell', h)^2 \equiv \ell'^2 \equiv 1 \pmod 8$. Thus, if $L$ is divisible by $3$, then we have $A\in \Z$. Otherwise, $\ell'$ is not divisible by $3$ and so $(\ell', h)^2 \equiv \ell'^2 \equiv 1 \pmod 3$. Therefore we have $A \in \Z$ in both cases, and hence the second condition is also satisfied. 
Lastly, the fourth condition is obviously satisfied. Thus, the result follows.
\end{proof}

\ms
For the proof of Theorem \ref{theorem: modular unit iff} and \ref{theorem: modular unit suff}, we need the following lemmas. 

\begin{lem}\label{lemma: lemma 8}
Let $G$ be the subgroup of $\G_0(N)$ generated by $\mat 1 1 0 1$ and $\mat 1 0 N 1$.
\begin{enumerate}[(i)]
\item
Let
\begin{equation*}
f(\tau)=\prod_{m \in \mathcal{D}_N} \prod_{h=0}^{\ell(m)-1} F_{m, h}(\tau)^{e_{m, h}} ~~\text{ for some } e_{m, h} \in \Z.
\end{equation*}
Suppose that the orders of $f$ at the cusps $0$ and $\infty$ are integers. Then for any $\g \in \G_0(N)$, the value of the root of unity $\e$ in $f(\g \tau)=\e f(\tau)$ depends only on the right coset $G \g$ of $\g$ in $\G_0(N)$.
\item
Every right coset of $G$ in $\G_0(N)$ contains an element $\mat a b c d$ such that $24 \dd b$, $24N \dd c$ and $c>0$.
\end{enumerate}
\end{lem}
\begin{proof}
Part (i) is the content of Lemma 7 in \cite{WY20}.\footnote{Let $\g=\mat 1 1 0 1$ and $\g'=\mat 1 0 N 1$. Since $\g$ (resp. $\g'$) is a generator of the isotropy subgroup of the cusp $\infty$ (resp. $0$), $f(\g \tau)=f(\tau)$ (resp. $f(\g' \tau)=f(\tau)$) for all $\tau$ if and only if the order of $f$ at the cusp $\infty$ (resp. $0$) is an integer.}
 In Lemma 8 of \textit{op. cit.}, we have shown that every right coset of $G$ in $\G_0(N)$ contains an element $\mat a b c d$ such that $24N \dd c$. Now we have
\begin{equation*}
\mat 1 k 0 1 \mat a b c d= \mat {a+kc} {b+kd} c d.
\end{equation*}
Since $24 \dd c$, we have $(d, 6)=1$ and hence there exists an integer $k$ such that $24 \dd (b+kd)$. Furthermore, we have
\begin{equation*}
\mat 1 0 {24kN} 1 \mat a b c d=\mat a b {c+24akN} {d+24bkN}.
\end{equation*}
Thus, there is an integer $k$ such that $c+24akN>0$. This proves (ii).
\end{proof}

\begin{lem}\label{lemma: lemma 10}
Let $m \in \mathcal{D}_N$ such that $m''\neq 2$ and let $a$ be an odd integer prime to $m''$. 
If $L$ is odd, then we have
\begin{equation*}
\prod_{\a \in S_{m''}} (-1)^{\gauss{\a a/{m''}}}=\begin{cases}
(-1)^{\frac{(a-1)\p(m'')}{4}} & \text{ if }~m'' \text{ is even},\\
\leg {a}{p} & \text{ if } m''=p^k \text{ for some odd prime } p,\\
1 & \text{ otherwise}.
\end{cases}
\end{equation*}
\end{lem}
\begin{proof}
For $\a \in S_{m''}$, let $\a^*$ be the unique element in $S_{m''}$ such that $\a a \equiv \a^* \pmod {m''}$ or $\a a \equiv -\a^* \pmod {m''}$. Let $s$ be the number of $\a$ in $S_{m''}$ such that the latter case occurs. Then we have
\begin{equation*}
\begin{split}
\sum_{\a \in S_{m''}} \gauss {\frac{\a a}{m''}} &=\sum_{\a : \a a \equiv \a^* \pmod {m''}} \frac{\a a-\a^*}{m''}+\sum_{\a : \a a \equiv -\a^* \pmod {m''}} \frac{\a a-(m''-\a^*)}{m''}\\
&=-s+(a-1)\sum_{\a \in S_{m''}} \frac{\a}{m''}+2\sum_{\a : \a a \equiv -\a^* \pmod {m''}} \frac{\a^*}{m''}.
\end{split}
\end{equation*}
So we have $m''_0\sum_{\a \in S_{m''}} \gauss {\frac{\a a}{m''}}=-m''_0 s+2^{1-v_2(m'')}C$, where $m''_0$ is the odd part of $m''$ and 
\begin{equation*}
C=\frac{a-1}{2} \sum_{\a \in S_{m''}} \a + \sum_{\a : \a a \equiv -\a^* \pmod {m''}} \a^*
\end{equation*}
is an integer. Since we assume that $N$ is not divisible by $4$, $v_2(m'')$ is either $0$ or $1$. Hence we have
\begin{equation*}
\prod_{\a \in S_{m''}} (-1)^{\gauss{\a a/{m''}}} =\begin{cases}
(-1)^s & \text{ if $m''$ is odd},\\
(-1)^{s+C} & \text{ if $m''$ is even}.
\end{cases}
\end{equation*}

First, suppose that $m''$ is odd. Note that
\begin{equation*}
\prod_{\a \in S_{m''}} (\a a) \equiv (-1)^s \prod_{\a \in S_{m''}} \a^* \equiv (-1)^s \prod_{\a \in S_{m''}} \a \pmod {m''}.
\end{equation*}
Thus, we have $(-1)^s \equiv a^{\p(m'')/2} \pmod {m''}$. If $m''$ has at least two distinct prime divisors, then the exponent of the group $(\zmod {m''})^\times$ is a divisor of $\p(m'')/2$, and hence $(-1)^s =1$. If $m''=p^k$ is an odd prime power, then $a^{\p(m'')/2} \equiv \leg a p \pmod {m''}$ by (a generalization of) Euler's theorem. 

Next, suppose that $m''$ is even. Then for any $\a \in S_{m''}$ we have $\a \equiv 1 \pmod 2$.
Thus, we have
\begin{equation*}
C \equiv \frac{(a-1)}{2} \times \frac{\p(m'')}{2}+s \pmod 2.
\end{equation*}
This completes the proof.
\end{proof}

\begin{prop}\label{proposition: proposition 11}
Let
\begin{equation*}
f(\tau)=\prod_{m \in \mathcal{D}_N} \prod_{h=0}^{\ell(m)-1} F_{m, h}(\tau)^{e_{m, h}} \quad \text{ for some } ~e_{m, h} \in \Z
\end{equation*}
and
\begin{equation*}
\G:=\left\{ \mat a b c d \in \G_0(N) : a \equiv d \equiv \pm 1 \pmod {N/L} \right\}.
\end{equation*}
Suppose that $L$ is odd. Then $f$ is a modular function on $X_\G$ if and only if all the following conditions are satisfied:
\begin{enumerate}
\item
the order of $f$ at $\infty$ is an integer;
\item
the order of $f$ at $0$ is an integer; 
\item 
the order of $f$ at the cusp $\vect 1 {{N_0}}$ is an integer, where $N_0$ is the odd part of $N$;
\item (the mod $L$ condition) we have
\begin{equation*}
\sum_{m \in \mathcal{D}_N} m\p(m'') \sum_{h=0}^{\ell(m)-1} he_{m, h}\equiv 0 \pmod L.
\end{equation*}
\end{enumerate}
\end{prop}
\begin{proof}
The first three conditions are clear, which we assume from now on. 

First, we consider the case where $N=2N_0$ for some odd integer $N_0$. 
For simplicity, for $n \in \{0, 1\}$ let
\begin{equation*}
\mathcal{D}_n:=\{ m \in \mathcal{D}_N : m''=2^n p^r \text{ for some } p\equiv 3\pmod 4 \text{ and  an integer } r\geq 1\}. 
\end{equation*}
Then, for any $m\in \mathcal{D}_N$ different from $N_0$, we have $\p(m'')\equiv 0 \pmod 2$; and moreover 
\begin{equation}\label{equation: 01}
\prod_{p \mid m''} (1-p) \equiv 2 \pmod 4 \iff \p(m'') \equiv 2 \pmod 4 \iff m \in \mathcal{D}_0 \cup \mathcal{D}_1.
\end{equation}
Since $L$ is odd, $m$ is odd if and only if $m''$ is even. 
Thus by Corollary \ref{corollary: order at special cusps}, the first condition implies that 
\begin{equation*}
e_{N_0, 0} \equiv 0 \pmod 2 \qa \frac{e_{N_0, 0}}{2}+\sum_{m \in \mathcal{D}_1} \sum_{h=0}^{\ell(m)-1} e_{m, h} \equiv 0 \pmod 2,
\end{equation*}
and the third condition implies that
\begin{equation*}
e_{N_0, 0} \equiv 0 \pmod 2  \qa \frac{e_{N_0, 0}}{2}+\sum_{m \in \mathcal{D}_0 \cup \mathcal{D}_1} \sum_{h=0}^{\ell(m)-1} e_{m, h} \equiv 0 \pmod 2.
\end{equation*}
Combining these, we get 
\begin{equation}\label{equation: 02}
e_{N_0, 0} \equiv 0 \pmod 2  \qa \sum_{m \in \mathcal{D}_0} \sum_{h=0}^{\ell(m)-1} e_{m, h} \equiv 0 \pmod 2.
\end{equation}

Let $m \in \mathcal{D}_N$ and $\g=\mat a b c d \in \G$. 
Replacing $\g$ by $-\g$ if necessary, we assume that $a \equiv d \equiv 1 \pmod {N/L}$. Note that multiplication by $\mat 1 1 0 1$ or $\mat 1 0 N 1$ on the left of $\g$ does not change the residue class of $a$ modulo $N$. 
Therefore by Lemma \ref{lemma: lemma 8}, we may further assume that $24 \dd b$, $24N \dd c$ and $c>0$. 
If $m''=2$, i.e., $m=N_0$, then by Remark \ref{remark: root of unity for m''=2} we have
\begin{equation*}
F_{N_0, 0}(\g \tau)=(-1)^{\frac{a^2-1}{8}} F_{N_0, 0}(\tau).
\end{equation*}
Assume that $m''\neq 2$, i.e., $m\neq N_0$. Then by Lemma \ref{lemma: lemma 9}, we have\footnote{The factor $e^{2\pi i B}$ disappears since $d\equiv 1 \pmod {N/L}$ and $d, \ell$ are both odd.}
\begin{equation*}
F_{m, h}(\g \tau)=(-1)^{\frac{(d-1)\p(m'')}{4}} \prod_{\a \in S_{m''}} E_{\a am\ell, \d dhN'}(N'\tau).
\end{equation*}
Let $k=\frac{a-1}{(N/L)} \in \Z$.\footnote{Note that since $a$ can be any integer congruent $1$ modulo ${N/L}$, $k$ can be any integer. Also, since multiplication by $\mat 1 1 0 1$ or $\mat 1 0 N 1$ on the left of $\g$ does not change the residue class of $a$ modulo $N$, the residue class of $k$ modulo $L$ does not change.} 
Since $N/L$ is a multiple of $m''$ for any $m$, we have $k_m:=\frac{a-1}{m''} \in \Z$. By (\ref{equation: 1}), we have
\begin{equation*}
E_{\a am\ell, \d dhN'}(N'\tau)=(-\z_\ell^{-\d h})^{\a k_m} E_{\a m \ell, \d h N'}(N'\tau).
\end{equation*}
If $k_m$ is odd, then $m''$ must be even and so $\a$ is odd for any $\a \in S_{m''}$. Since $\a \d \equiv 1 \pmod \ell$ and $k_m=\frac{km\ell}{L}$, we have
(whether $k_m$ is odd or even)
\begin{equation*}
\begin{split}
F_{m, h}(\g \tau)&=(-1)^{\frac{k_m\p(m'')}{2}} (-1)^{\frac{(d-1)\p(m'')}{4}} \z_{\ell}^{-h k_m \p(m'')/2} F_{m, h}(\tau)\\
&=(-1)^{(\frac{\p(m'')}{2})(\frac{km\ell}{L}+\frac{d-1}{2})}\z_L^{-\frac{kmh\p(m'')}{2}} F_{m, h}(\tau).
\end{split}
\end{equation*} 
Since $\ell$ and $L$ are both odd, we have $\frac{km\ell}{L} \equiv km \pmod 2$. Also since $e_{N_0, 0}$ is even by (\ref{equation: 02}), we have
\begin{equation*}
f(\g \tau)=\e \mu f(\tau),
\end{equation*}
where
\begin{equation*}
\e=\prod_{m \in \mathcal{D}_N, m\neq N_0} \prod_{h=0}^{\ell(m)-1} (-1)^{(\frac{\p(m'')}{2})(km+\frac{d-1}{2})e_{m, h}}
\end{equation*}
and
\begin{equation*}
\mu=\prod_{m \in \mathcal{D}_N, m\neq N_0} \prod_{h=0}^{\ell(m)-1} \z_L^{-\frac{kmh\p(m'')e_{m, h}}{2}}.
\end{equation*}

Since $\e \in \{ \pm 1\}$ and $\mu$ is an $L$-th root of unity with odd $L$, $\e\mu=1$ if and only if $\e$ and $\mu$ are both equal to $1$. 

First, if $a \equiv d\equiv 1 \pmod 4$, then $k$ must be even, and hence $\e=1$. If $a\equiv d\equiv  3\pmod 4$, then $k$ is odd. Therefore by (\ref{equation: 01}), $\e=1$ if and only if $\sum_{m \in \mathcal{D}_0} \sum_{h=0}^{\ell(m)-1} e_{m, h} \equiv 0 \pmod 2$, which follows by (\ref{equation: 02}).

Next, $\mu=1$ if and only if 
\begin{equation*}
k\sum_{m\in \mathcal{D}_N, m \neq N_0} m \p(m'') \sum_{h=0}^{\ell(m)-1} he_{m, h} \equiv 0 \pmod L.
\end{equation*}
Since $k$ can be arbitrary if we vary $\g \in \G$, the result follows in this case.

If $N$ is odd, the argument is much simpler as $m''\neq 2$. We leave the details to the readers.
\end{proof}

Now, we are ready to prove Theorem \ref{theorem: modular unit iff}.
\begin{proof}[Proof of Theorem \ref{theorem: modular unit iff}]
We use the same notation as in the proof of Proposition \ref{proposition: proposition 11}. 
Also, we assume that $N=2N_0$ for some odd integer $N_0$ and leave the simpler case for odd $N$ to the readers.

By Proposition \ref{proposition: proposition 11}, in order for $f$ to be a modular function on $\G_0(N)$, conditions (1), (2), (3) and (4) must hold, which we assume from now on. In particular, we have (\ref{equation: 02}). 

Let $\g=\mat a b c d \in \G_0(N)$. As before, by Lemma \ref{lemma: lemma 8} we may assume that $24 \dd b$, $24N \dd c$ and $c>0$. By Lemma \ref{lemma: lemma 9} and (\ref{equation: 1}), for any $m \in \mathcal{D}_N$  with $m''\neq 2$, we have
\begin{equation*}
\begin{split}
F_{m, h}(\g \tau)&=(-1)^{\frac{(d-1)\p(m'')}{4}} e^{2\pi i B} \prod_{\a \in S_{m''}} E_{\a a m\ell, \d d h N'}(N'\tau)\\
&=(-1)^{\frac{(d-1)\p(m'')}{4}} e^{2\pi i B} \prod_{\a \in S_{m''}} (-\z_\ell^{-\d dh})^{\gauss{\a a/{m''}}}E_{\a^* m\ell, \d^* h N'}(N'\tau)\\
&=(-1)^{\frac{(d-1)\p(m'')}{4}} e^{2\pi i B} \left(\prod_{\a \in S_{m''}} (-\z_\ell^{-\d dh})^{\gauss{\a a/{m''}}} \right) F_{m, h}(\tau),\\
\end{split}
\end{equation*}
where $B=\frac{h(1-d)}{2\ell}\sum_{\a \in S_{m''}} \d$ and for $\a \in S_{m''}$, we let $\a^*$ be the element in $S_{m''}$ such that $\a a \equiv \a^* \pmod {m''}$ or $\a a \equiv -\a^* \pmod {m''}$, and $\d^*\in \Z$ satisfies $\a^* \d^* \equiv 1 \pmod {m''}$. Since $a\equiv d \pmod 4$ and $e_{N_0, 0} \equiv 0 \pmod 2$, by Remark \ref{remark: root of unity for m''=2} and Lemma \ref{lemma: lemma 10} we have
\begin{equation*}
f(\g \tau)=\e_1 \e_2 \mu f(\tau),
\end{equation*}
where
\begin{equation*}
\begin{split}
\e_1&=\prod_{m \in \mathcal{D}_N, m'' \not\in 2\Z} \prod_{h=0}^{\ell(m)-1}(-1)^{(\frac{a-1}{2})(\frac{\p(m'')}{2})e_{m, h}}, \\
\e_2&=\prod_{m:m''=p^r, p\neq 2}\prod_{h=0}^{\ell(m)-1} \leg a p^{e_{m, h}}
\end{split}
\end{equation*}
and $\mu$ is an $L$-th root of unity depending on $\g$. As seen in the proof of Proposition \ref{proposition: proposition 11}, since $L$ is odd, 
$\e_1\e_2\mu=1$ if and only if $\e_1\e_2=1$ and $\mu=1$. Also, we have $\e_1=1$ by (\ref{equation: 01}) and (\ref{equation: 02}). Furthermore, for each prime divisor $p$ of $N$, there is always an integer $a$ such that $(a, 24N)=1$, $\leg a p=-1$ and $\leg a q=1$ for all odd prime divisors $q$ of $N$ different from $p$. Thus, $\epsilon_2=1$ if and only if the last condition is satisfied. 

Now the key observation is that $\g^{\p(N/L)} \in \G$. Thus, by the condition (4) and Proposition \ref{proposition: proposition 11} we have
\begin{equation*}
f(\g^{\p(N/L)} \tau)=\mu^{\p(N/L)} f(\tau)=f(\tau)
\end{equation*}
and so $\mu^{\p(N/L)}=1$. Since $(\p(N/L), L)=1$, we get $\mu=1$, as desired. This completes the proof.
\end{proof}

\begin{proof}[Proof of Theorem \ref{theorem: modular unit suff}]
Suppose that all the conditions are fulfilled. By the same argument as in the proof of Theorem \ref{theorem: modular unit iff}, it suffices to show that for each $\g=\mat a b c d \in \G_0(N)$ with $24 \dd b$, $24N \dd c$ and $c > 0$, we have
$f^L(\g \tau)=f^L(\tau)$.
Again, by the same argument as above, we have
\begin{equation*}
f(\g \tau)=\mu f(\tau)
\end{equation*}
for some $L$-th root of unity $\mu$. Thus, $f^L(\g \tau)=\mu^L f^L(\tau)=f^L(\tau)$ and the result follows.
\end{proof}

\ms
\section{Proof of Theorem $1.4$}\label{section4}
In this section, we prove Theorem \ref{theorem: main theorem}.
For a positive integer $N$, we note that the following two statements are equivalent:
\begin{enumerate}
\item
We have $\scC(N)=\scC_N(\Q)$.
\item
For any primes $q$, we have
\begin{equation}\label{equation: conjecture}
\scC(N)[q^\infty] = \scC_N(\Q)[q^\infty],
\end{equation}
where $A[q^\infty]$ denotes the $q$-primary subgroup of $A$.
\end{enumerate}

As mentioned before, a cuspidal divisor $D$ on $X_0(N)$ is rational if and only if $\s(D)=D$ for all $\s \in \Gal(\Q(\z_L)/\Q)$. 
\begin{lem}\label{lemma: averaging case}
Suppose that $q$ does not divide $\p(L)$. Then we have
\begin{equation*}
\scC(N)[q^\infty] = \scC_N(\Q)[q^\infty].
\end{equation*}
\end{lem}
\begin{proof}
Let $D$ be a cuspidal divisor of degree $0$ such that $[D] \in \scC_N(\Q)[q^\infty]$.
Since $\scC(N) \subset \scC_N(\Q)$, it suffices to show that $D$ is linearly equivalent to a rational cuspidal divisor.
Assume that $[D]$ has order $q^r$ in $J_0(N)$, so we have $q^r D \sim 0$. Since $q$ does not divide $\p(L)$, there is an integer $k$ such that
$k\p(L) \equiv 1 \pmod {q^r}$. Let $a=\frac{k\p(L)-1}{q^r} \in \Z$ and let
\begin{equation*}
D'=\sum_{\s \in \Gal(\Q(\z_L)/\Q)} \s(D),
\end{equation*}
which is a rational cuspidal divisor by its construction. Since $\s(D) \sim D$ for any $\s \in \Gal(\Q(\z_L)/\Q)$, we have
$D' \sim \p(L) D$. Thus, we have
\begin{equation*}
D \sim (1+aq^r)D=k \p(L) D \sim kD'.
\end{equation*}
This completes the proof.
\end{proof}

From now on, let $N=p^2M$, where $p$ is a prime and $M$ is a squarefree integer.
Note that if $p=2$, then all the cusps of $X_0(N)$ are defined over $\Q$, so Theorem \ref{theorem: main theorem} obviously holds. Thus, we henceforth assume that $p$ is an odd prime. 
In this case, $L=p$ is odd and $\p(L)=p-1$ is relatively prime to $p$.
So in order to prove Theorem \ref{theorem: main theorem}, it suffices to prove (\ref{equation: conjecture}) for a prime $q$ different from $p$ by Lemma \ref{lemma: averaging case}. Thus, we further assume that $q$ is a prime different from $p$.
Since $L=p$ is an odd prime, $\ell(m)=p$ if $\ell(m)\neq 1$. Let $\mathcal{D}_N^*=\{ m \in \mathcal{D}_N : \ell(m)=p \}$.

Let $D$ be a cuspidal divisor of degree $0$ such that $[D] \in \scC_N(\Q)[q^\infty]$. As above, it suffices to show that $D$ is linearly equivalent to a rational cuspidal divisor.
Assume that $[D]$ has order $q^r$ in $J_0(N)$ and let $f$ be a modular function on $X_0(N)$ such that $q^r D=\tn{div} (f)$. By Theorem \ref{theorem: theorem 1}, (up to constant) we have
\begin{equation*}
f(\tau)=\prod_{m \in \mathcal{D}_N} \prod_{h=0}^{\p(\ell(m))-1} F_{m, h}(\tau)^{e_{m, h}} \quad \text{ for some } e_{m, h} \in \Z.
\end{equation*}
We claim that
\begin{enumerate}[(a)]
\item
$q^r \dd e_{m, h}$ for any $m \in \mathcal{D}_N^*$ and $1\leq h \leq p-2$; and
\item
the product
\begin{equation*}
F=\prod_{m \in \mathcal{D}_N^*} \prod_{h=1}^{p-2} \left(\frac{F_{m, h}}{F_{m, 0}}\right)^{npq^{-r}e_{m, h}}
\end{equation*}
is a modular function on $X_0(N)$, where $n=(3, p)$.
\end{enumerate}

Assume that the two claims hold. Let $D'=npD-\tn{div}(F)$. Then we assert that $\s(D')=D'$ for any $\s \in \Gal(\Q(\z_L)/\Q)$. For simplicity, let 
\begin{equation*}
G=\prod_{m \in \mathcal{D}_N^*} \prod_{h=1}^{p-2} \left( \frac{F_{m, h}}{F_{m, 0}}\right)^{e_{m, h}}.
\end{equation*}
Then we easily have $F^{q^r}=G^{np}$ and so
\begin{equation*}
q^r D'=q^r npD-q^r\tn{div}(F)=\tn{div} (f^{np}) - \tn{div}(F^{q^r})=\tn{div} \left(\left( f/G\right)^{np}\right).
\end{equation*}
By direct computation, we have
\begin{equation*}
f/G=\prod_{m \in \mathcal{D}_N\sm \mathcal{D}_N^*} F_{m, 0}^{e_{m, 0}}\prod_{m\in \mathcal{D}_N^*} F_{m, 0}^{A_m} ,
\end{equation*}
where $A_m:=\sum_{h=0}^{p-2} e_{m, h}$. 
Hence by Corollary \ref{corollary: rationality}, the assertion follows.

Since $np$ is relatively prime to $q^r$, there is an integer $k$ such that $npk \equiv 1 \pmod {q^r}$. Let $a=\frac{npk-1}{q^r} \in \Z$. Then we have
\begin{equation*}
D \sim (1+a q^r)D=np kD=k(D'+\tn{div}(F)) \sim k D'
\end{equation*}
as desired. So it suffices to prove two claims above.

\ms
\noindent {\it Proof of} (a). Let $s \in \{1, \dots, p-1\}$. By Lemma \ref{lemma: galois action}, we have
\begin{equation*}
\s_s(q^r D)= \s_s(\tn{div} (f))=\tn{div}\left(\prod_{m \in \mathcal{D}_N}  \prod_{h=0}^{\p(\ell(m))-1} F_{m, sh}^{e_{m,h}}\right),
\end{equation*}
where the integer $sh$ in the subscript is understood to be taken modulo $p$. 
Since $\s_s(D)-D$ is assumed to be a principal divisor, the product
\begin{equation*}
\prod_{m \in \mathcal{D}_N^*} \prod_{h=1}^{p-2} \left(\frac{F_{m, sh}}{F_{m, h}}\right)^{e_{m, h}}
\end{equation*}
must be the $q^r$-th power of some modular unit. By Theorem \ref{theorem: variant of theorem 1}, this implies that
\begin{equation*}
\begin{cases}
q^r \dd e_{m, h}  & \text{ if } ~~h\equiv -s \pmod p \qqo sh \equiv -1 \pmod p,\\
q^r \dd (e_{m, s^*h}-e_{m, h}) & \text{ otherwise},
\end{cases}
\end{equation*}
where $s^* \in \Z$ such that $ss^* \equiv 1 \pmod p$. By taking $s \in \{2, \dots, p-1\}$, we conclude that $q^r \dd e_{m, h}$ for any $m \in \mathcal{D}_N^*$ and $1\leq h \leq p-2$. \qed

\ms
\noindent {\it Proof of} (b).  Since $q^{-r} e_{m, h} \in \Z$ by (a), the claim follows by Corollary \ref{corollary: modular function Gmh}. \qed


\section*{Acknowledgments}
J.-W. G. was supported by Grants 109-2115-M-002-017-MY2
of the Ministry of Science and Technology, Taiwan (Republic of China).  
Y.Y. was supported by Grants 109-2115-M-002-010-MY3 and
110-2115-M-002-010 of the Ministry of Science and Technology,
Taiwan (Republic of China).  
H.Y. was supported by National Research Foundation of Korea (NRF) grant funded by the Korea government (MSIT) 
(Nos. 2019R1C1C1007169 and 2020R1A5A1016126). 
M.Y. was supported by the National Research Foundation of Korea (NRF) grant funded by the Korea government (MSIT) (No. 2020R1C1C1A01007604), by Korea Institute for Advanced Study (KIAS) grant funded by the Korea government, and by Yonsei University Research Fund (2022-22-0125).

\end{document}